\documentclass[reqno]{amsart}

\usepackage[sorting=none, sorting=nyt, maxbibnames=99, backend=biber]{biblatex}

\renewbibmacro{in:}{}
\bibliography{references.bib}
\usepackage{amssymb}
\usepackage{calrsfs}
\usepackage{graphics,graphicx, mathrsfs}
\usepackage{enumerate}
\usepackage{enumitem}
\usepackage{textcomp}
\usepackage{url}
\usepackage{xcolor}
\definecolor{vio}{rgb}{0.54, 0.17, 0.89}
\usepackage[ruled, lined, linesnumbered, longend]{algorithm2e}
\usepackage{hyperref}
\usepackage{titlesec}

\newtheorem{theorem}{Theorem}[section]
\newtheorem{lemma}[theorem]{Lemma}

\newtheorem{proposition}[theorem]{Proposition}
\newtheorem{conjecture}[theorem]{Conjecture}
\newtheorem{corollary}[theorem]{Corollary}
\numberwithin{equation}{section}

\theoremstyle{remark}
\newtheorem*{remark}{Remark}

\titleformat{\section}
  {\normalfont\large\bfseries\centering}{\thesection}{1em}{}

\titleformat{\subsection}
  {\normalfont\bfseries}{\thesubsection}{1em}{}

\DeclareMathOperator{\PP}{\mathcal{P}}

\def\reals{\hbox{\rm I\kern-.18em R}}
\def\complexes{\hbox{\rm C\kern-.43em
\vrule depth 0ex height 1.4ex width .05em\kern.41em}}
\def\field{\hbox{\rm I\kern-.18em F}} 

\let\svthefootnote\thefootnote
\newcommand\freefootnote[1]{%
  \let\thefootnote\relax%
  \footnotetext{#1}%
  \let\thefootnote\svthefootnote%
}

\newenvironment{section*}[2][A]{
  \section*{#2}
  \renewcommand\thesection{#1}
  \setcounter{theorem}{0}}{}

\allowdisplaybreaks

\begin{document}

\title[The sum of a prime power and an almost prime]{The sum of a prime power and an almost prime}

\author{Daniel R. Johnston and Simon N. Thomas}
\address{School of Science, UNSW Canberra, Australia}
\email{daniel.johnston@unsw.edu.au}
\address{School of Mathematics and Physics, The University of Queensland}
\email{simon.thomas1@uq.edu.au}
\date\today

\begin{abstract}
    For any fixed $k\geq 2$, we prove that every sufficiently large integer can be expressed as the sum of a $k$th power of a prime and a number with at most $M(k)=6k$ prime factors. For sufficiently large $k$ we also show that one can take $M(k)=(2+\varepsilon)k$ for any $\varepsilon>0$, or $M(k)=(1+\varepsilon)k$ under the assumption of the Elliott--Halberstam conjecture. Moreover, we give a variant of this result which accounts for congruence conditions and strengthens a classical theorem of Erd\H{o}s and Rao. The main tools we employ are the weighted sieve method of Diamond, Halberstam and Richert, bounds on the number of representations of an integer as the sum of two $k$th powers, and results on $k$th power residues. We also use some simple computations and arguments to conjecture an optimal value of $M(k)$, as well as a related variant of Hardy and Littlewood's Conjecture H.
\end{abstract}

\maketitle

\freefootnote{\textit{Corresponding author}: Daniel Johnston (daniel.johnston@unsw.edu.au).}
\freefootnote{\textit{Affiliation}: School of Science, The University of New South Wales Canberra, Australia.}
\freefootnote{\textit{Key phrases}: Goldbach's conjecture, Chen's theorem, sieve methods, almost primes.}
\freefootnote{\textit{2020 Mathematics Subject Classification}: 11N36, 11P32 (Primary) 11A15, 11D59 (Secondary)}

\section{Introduction}
In this paper we are interested in representing integers $N\geq 1$ as
\begin{equation}\label{mainrepeq}
    N=p^k+\eta,
\end{equation}
where $p$ is prime, $k\geq 1$ is fixed, and $\eta$ has a uniformly bounded number of prime factors. That is, $\Omega(\eta)\leq M(k)$, where $\Omega(\eta)$ is the number of prime factors of $\eta$ counting multiplicity, and $M(k)$ is only a function of $k$. Historically, the case $k=1$ has been of significant interest due to its close link to Goldbach's conjecture. Here, we focus on the case $k\geq 2$.

\subsection{Overview of existing literature}\label{overviewsect}
We refer the reader to any of the recent surveys \cite{kumchev2005invitation,vaughan2016goldbach,li2023some} on Goldbach's conjecture and additive number theory for a broad history of similar problems and results. However, in what follows, we highlight some of the most relevant results in our setting.

To begin with, we recall that in 1938 Hua \cite{hua1938some} proved the following as an application of the circle method.
\begin{theorem}[Hua]\label{huathm}
    Let $k\geq 1$, $X\geq 1$ and consider positive integers $N\leq X$ with  $N\not\equiv 1\pmod{q}$ for any prime $q$ with $q-1\mid k$. Then, with at most $o(X)$ exceptions, $N$ can be expressed as
    \begin{equation*}
        N=p_1^k+p_2,
    \end{equation*}
    where $p_1$ and $p_2$ are prime.
\end{theorem}
In particular, Hua's result gives that for all $k\geq 1$ and \emph{almost all} $N\geq 1$ (subject to suitable congruence conditions), $N$ can be represented as \eqref{mainrepeq}, with $\Omega(\eta)~=~1$. Note that Hua's result has since been refined, with stronger bounds on the exceptional set of $N$ (see \cite[Satz 2]{schwarz1961darstellung} and \cite{bauer1999exceptional}).

In what follows we only consider results for \emph{all} sufficiently large $N$. First we recall the celebrated result of Chen \cite{chen1973} for $k=1$ and $N$ even, along with a recent variant for odd $N$ due to Li \cite{li2019representation}. Both of these results were proven using techniques from sieve theory.

\begin{theorem}[Chen]\label{chenthm}
    Every sufficiently large even number $N$ can be expressed as
    \begin{equation}\label{Cheneq}
        N=p+\eta,
    \end{equation}
    where $p$ is prime and $\eta>0$ has at most 2 prime factors counting multiplicity.
\end{theorem}
\begin{theorem}[Li]\label{lithm}
    Every sufficiently large odd number $N$ can be expressed as
    \begin{equation}\label{Lieq}
        N=p+2\eta,
    \end{equation}
    where $p$ is prime and $\eta>0$ has at most 2 prime factors counting multiplicity.
\end{theorem}
\noindent
Here, the factor of 2 in \eqref{Lieq} is to force $p+2\eta$ to be odd; provided of course $p\neq 2$.

Combining Theorems \ref{chenthm} and \ref{lithm}, we have that every large $N\geq 1$ can be represented as the sum of a prime and a number with at most 3 prime factors. Any further lowering of the number of prime factors in Chen or Li's results appears out of reach, so we do not attempt this here.

Now, representing $N$ of the form \eqref{mainrepeq} when $k\geq 2$ is naturally more difficult, owing to sparseness of the set of $k$th prime powers, and additional ``forced" prime factors of $N-p^k$. Thus in this setting, much less is known. However, the following classical result, due to Erd\H{o}s \cite{erdos1935representation} and Rao \cite{rao1940representation}, can be viewed as partial progress towards representing large $N$ in the form \eqref{mainrepeq}.

\begin{theorem}[Erd\H{o}s--Rao]\label{erdraothm}
    Let $k\geq 2$. Then every sufficiently large integer $N$ such that
    \begin{equation}\label{erdraocong}
        N\not\equiv 1
        \begin{cases}
            \pmod{16},&\text{when $k=4$},\\
            \pmod{4},&\text{when $k=2$},
        \end{cases}
    \end{equation}
    can be expressed as
    \begin{equation*}
        N=p^k+\eta    
    \end{equation*}
    where $p$ is prime and $\eta>0$ is $k$th power-free.
\end{theorem}

Here, the condition that $\eta$ is $k$th power-free is much simpler than $\eta$ having a bounded number of prime factors. In particular, most integers ($>60\%$) are square-free (and thus $k$th power-free), whereas the set of integers with a bounded number of prime factors has a natural density of 0.

\subsection{Statement of results}
Our main result is as follows.
\begin{theorem}\label{mainthmsimple}
    Let $k\geq 2$. There exists an integer $M(k)$ such that every sufficiently large integer $N$ can be expressed as
    \begin{equation}\label{firstNeq}
        N=p^k+\eta
    \end{equation}
    where $p$ is prime and $\eta>0$ has at most $M(k)$ prime factors counting multiplicity. Here, $M(k)=6k$ is admissible for all even $k\geq 2$, and $M(k)=4k$ is admissible for all odd $k\geq 3$. In addition, for sufficiently large $k\geq k_{\varepsilon}$ one can set
    \begin{equation}\label{unconasym}
        M(k)=(2+\varepsilon)k,
    \end{equation}
    for any $\varepsilon>0$. Or, under assumption of the Elliott--Halberstam conjecture,
    \begin{equation}\label{ehasym}
        M(k)=(1+\varepsilon)k.
    \end{equation}
\end{theorem}
Similar to the statement of Li's theorem (Theorem \ref{lithm}), there may be forced divisors of $\eta$ depending on the congruence class of $N$. For example, if $k=2$ and $N\equiv 1\pmod{3}$, then $\eta=N-p^2\equiv 0\pmod{3}$ for all $p\neq 3$. We thus in fact prove the following more technical result, which accounts for congruence conditions on $N$, and from which Theorem \ref{mainthmsimple} will follow.

\begin{theorem}\label{mainthmtech}
    Fix $k\geq 1$ and let $N$ be sufficiently large. Let $q_1,q_2,\ldots,q_{\ell_k}$ be all the primes such that $q_i-1\mid k$. For each $i\in\{1,2,\ldots,\ell_k\}$, let $m_i\geq 0$ be the largest integer such that $N\equiv 1\pmod{q_i^{m_i}}$ and 
    \begin{align}
        q_i^{m_i-1}(q_i-1)&\mid k,\quad\text{if $q_i\neq 2$ or $m_i\leq 2$,}\label{qcon1}\\
        2^{m_i-2}&\mid k,\quad \text{if $q_i=2$, $k$ is even and $m_i\geq 3$.}\label{qcon2}
    \end{align} 
    Then, there exists some $\widetilde{M}(k)>0$ (independent of $N$) such that
    \begin{equation}\label{secondNeq}
        N=p^k+q_1^{m_1}q_2^{m_2}\cdots q_{\ell_k}^{m_{\ell_k}}\cdot \eta,
    \end{equation}
    where $\eta>0$ is $k$th power-free (provided $k\geq 2$), not divisible by any of the $q_i$ and has at most $\widetilde{M}(k)$ prime factors counting multiplicity. Here, $\widetilde{M}(k)=4k+1$ is admissible for all even $k\geq 2$, and $\widetilde{M}(k)=4k-1$ is admissible for all odd $k\geq 1$. For $1\leq k\leq 10$, smaller values of $\widetilde{M}(k)$ are also given in Table \ref{Mktable}. In addition, for sufficiently large $k\geq k_{\varepsilon}$ one can set
    \begin{equation}\label{unconasym2}
        \widetilde{M}(k)=(2+\varepsilon)k,
    \end{equation}
    for any $\varepsilon>0$. Or, under assumption of the Elliott--Halberstam conjecture,
    \begin{equation}\label{ehasym2}
        \widetilde{M}(k)=(1+\varepsilon)k.
    \end{equation}
\end{theorem}

\def\arraystretch{1.4}
\begin{table}[h]
\centering
\caption{Smaller values for $\widetilde{M}(k)$, appearing in Theorem \ref{mainthmtech}, for $1\leq k\leq 10$. Corresponding values of $M(k)$ in Theorem \ref{mainthmsimple} are also given, computed via Corollary \ref{tildeMcor}.}
\begin{tabular}{|c|c c|}
\hline
$k$ & $\widetilde{M}(k)$ & $M(k)$\\
\hline
1 & 2 & 3\\
2 & 8 & 12\\
3 & 11 & 12\\
4 & 17 & 23\\
5 & 16 & 17\\
\hline
\end{tabular}
\begin{tabular}{|c|c c|}
\hline
$k$ & $\widetilde{M}(k)$ & $M(k)$\\
\hline
6 & 25 & 31\\
7 & 20 & 21\\
8 & 30 & 37\\
9 & 29 & 30\\
10 & 35 & 40\\
\hline
\end{tabular}
\label{Mktable}
\end{table}

Theorem \ref{mainthmtech} relates more directly to the literature discussed in Section \ref{overviewsect}. Firstly, the case where $m_1=m_2=\cdots=m_{\ell_k}=0$ corresponds to the congruence condition in Hua's Theorem (Theorem \ref{huathm}) and recovers Chen's theorem (Theorem~\ref{chenthm}) when $k=1$. Secondly, if $k$ and $N$ are odd then $m_1=1$ and $m_2=\cdots=m_{\ell_k}=0$, thereby reducing \eqref{secondNeq} to $N=p^k+2\eta$. So in particular, Theorem \ref{mainthmtech} also recovers Li's theorem (Theorem \ref{lithm}).

In relation to Erd\H{o}s and Rao's result (Theorem \ref{erdraothm}), one sees that the product
\begin{equation}\label{qprod}
    q_1^{m_1}q_2^{m_2}\cdots q_{\ell_k}^{m_{\ell_k}}
\end{equation}
appearing in \eqref{secondNeq} is $k$th power-free if and only if the condition \eqref{erdraocong} is satisfied. Since $\eta$ in \eqref{secondNeq} is $k$th power-free and coprime to the product \eqref{qprod}, Theorem \ref{mainthmtech} is also a direct strengthening of Theorem \ref{erdraothm}.

To prove Theorem \ref{mainthmtech} we apply the weighted Diamond--Halberstam--Richert (DHR) sieve, combined with bounds on the number of representations of an integer as the sum of two $k$th powers, and results on $k$th power residues. As in many sieve-theoretic applications, we also use the Bombieri--Vinogradov theorem.

With a more involved argument, it seems likely that our values of $\widetilde{M}(k)$ and $M(k)$ could be lowered even further. In this regard, we note that recent (unpublished) work by Lichtman \cite[Theorem 1.7]{lichtman2023primes} improves upon a variant of the Bombieri--Vinogradov theorem applicable to Goldbach-like problems. In addition, one could try constructing a hybrid sieve, similar to that due to Irving \cite{irving2015almost}, in an attempt to improve upon our results for large $k$.  

\subsection{Paper outline}
An outline of the rest of the paper is as follows. In Section \ref{litsect} we state some preliminary results from the literature. This includes a discussion of $k$th power residues, representations of integers as the sum of two $k$th powers, the DHR weighted sieve, and the Bombieri--Vinogradov theorem. In Section \ref{techsect} we prove Theorem \ref{mainthmtech}. Using Theorem \ref{mainthmtech} we then prove Theorem \ref{mainthmsimple} in Section \ref{simplesect}. Finally, in Section \ref{furthersect} we provide simple computations and arguments that allow us to conjecture an optimal value of $\widetilde{M}(k)$. Here, we also discuss the case where $N$ is not a $k$th power, and present a related variant of Hardy and Littlewood's ``Conjecture H".

\section{Preliminary results}\label{litsect}
In this section we recall some results from the literature that will be used to form the basis of our proof of Theorem \ref{mainthmtech}.

\subsection{Results on $k$th power residues}\label{kthressect}
We begin by discussing some results on $k$th power residues. This will motivate the rather technical statement of Theorem \ref{mainthmtech}.

So, suppose $k\geq 1$ is fixed. As mentioned in the introduction, there exist moduli $Q$, such that for certain congruence classes of $N$, we have 
\begin{equation}\label{npkexeq}
    N-p^k\equiv 0\pmod{Q}
\end{equation}
for all but finitely many primes $p$. By Dirichlet's theorem for primes in arithmetic progressions, for any modulus $Q\geq 1$ there exist infinitely many primes with $p\equiv 1$ (mod $Q$) and thus $p^k\equiv 1\pmod{Q}$. Hence, the only way for \eqref{npkexeq} to occur is if $N\equiv 1$ (mod $Q$) and each $p$ with $(p,Q)=1$ satisfies $p^k\equiv 1$ (mod $Q$). In other words, we are interested in classifying all $Q$ such that the number of solutions to $x^k\equiv 1$ (mod $Q$) is equal to $\varphi(Q)$, where $\varphi$ is the Euler totient function. Here we note that it suffices to consider the case where $Q$ is a prime power, since then we can extend to all moduli $Q$ using the Chinese remainder theorem. To solve this problem, we appeal to the following classical result.

\begin{lemma}[{See e.g.\ \cite[p.\ 45]{ireland1990classical}}]\label{localoblem}
    Let $k\geq 1$, $q$ be prime, $m\geq 1$ and $a$ be a $k$th power residue mod $q^m$. Then, the number of solutions to $x^k\equiv a$ (mod $q^m$) for $x\in\mathbb{Z}/q^m\mathbb{Z}$ is
    \begin{equation*}
        \gcd(k,\varphi(q^m))
    \end{equation*}
    if $q$ is odd and
    \begin{equation*}
        \begin{cases}
            \gcd(k,2^{m-1}),& m\leq 2,\\
            1,&m\geq 3\text{ and $k$ odd},\\
            2\cdot\gcd(k,2^{m-2}),&m\geq 3\text{ and $k$ even},\\
        \end{cases}
    \end{equation*}
    if $q=2$.
\end{lemma}

Noting that $\varphi(q^m)=q^{m-1}(q-1)$, it follows from Lemma \ref{localoblem} that the congruence $x^k\equiv 1$ (mod $q^m$) has $\varphi(q^m)$ solutions if and only if
\begin{align}
    q^{m-1}(q-1)&\mid k,\quad\text{if $q\neq 2$ or $m\leq 2$}\label{qeq1},\\
    2^{m-2}&\mid k,\quad \text{if $q=2$, $k$ is even and $m\geq 3$.}\label{qeq2}
\end{align}
In particular, by the Chinese remainder theorem, we have the following result.
\begin{proposition}\label{Qprop}
    For any $k\geq 1$, let $\mathcal{Q}=\{q_1,q_2,\ldots,q_{\ell}\}$ be a finite set of primes and $m_1,m_2,\ldots,m_{\ell}$ be nonnegative integers. If 
    \begin{equation*}
        Q=\prod_{1\leq i\leq\ell}q_i^{m_i}
    \end{equation*}
    then every prime $x$ with $x\notin\mathcal{Q}$ satisfies $x^k\equiv 1\pmod{Q}$ if and only if each pair $(q_i,m_i)$ satisfies \eqref{qeq1} and \eqref{qeq2}.
\end{proposition}
This explains the statement of Theorem \ref{mainthmtech}, whereby all $q^m$ satisfying \eqref{qeq1} and \eqref{qeq2} with $N\equiv 1$ (mod $q^m$) appear as factors of $N-p^k$ in \eqref{secondNeq}.

\subsection{Representations of $N$ as the sum of two $k$th powers}
Let $k\geq 2$. In order to prove the $\eta$ is $k$th power-free in Theorem \ref{mainthmtech}, we require sufficiently strong bounds on the number of solutions $(a,b)$ to the equation
\begin{equation}\label{sumkthpowers}
    N=a^k+b^k,
\end{equation}
where $(a,N)=(b,N)=1$ and $N\to\infty$. Such bounds have been well-studied in the literature and are often related to $d(N)$, the number of divisors of $N$ and $\omega(N)$, the number of distinct prime factors of $N$. We thus first recall standard bounds for $d(N)$ and $\omega(N)$, which will be used extensively throughout this paper.
\begin{lemma}[{\cite[Theorem 1]{nicolas1983majorations} and \cite[Th\'eor\`eme 11]{robin1983estimation}}]\label{divisorlem}
    For all $N\geq 3$,
    \begin{align*}
        2\leq d(N)\leq N^{1.07/\log\log N},\\
        1\leq\omega(N)\leq 1.3841\frac{\log N}{\log\log N}.
    \end{align*}
    In particular, $\omega(N)=o(\log N)$ and for any $\varepsilon>0$ we have $d(N)=O(N^{\varepsilon})$.
\end{lemma}
We now give an upper bound for the number of solutions to \eqref{sumkthpowers}.
\begin{proposition}\label{thueprop}
    For any $\varepsilon>0$, the number of solutions $r_k(N)$ to \eqref{sumkthpowers} satisfies 
    \begin{equation*}
        r_k(N)=O_k(N^{\varepsilon}).
    \end{equation*}
\end{proposition}
\begin{proof}
    First, as proven by Estermann \cite{estermann1931einige},
    \begin{equation*}
        r_2(N)\leq 2d(N),
    \end{equation*}
    which is $O(N^{\varepsilon})$ by Lemma \ref{divisorlem}. For $k\geq 3$, \eqref{sumkthpowers} is a special case of the Thue equation. Here it suffices to consider $(a,b)=1$ since otherwise $N$ would share a factor with $a$ and $b$. Thus, by the main theorem of \cite{bombieri1987thue}, the equation \eqref{sumkthpowers} has
    \begin{equation*}
        r_k(N)\ll k^{1+{\omega(N)}}
    \end{equation*}
    solutions, which again reduces to $O_k(N^{\varepsilon})$ by Lemma \ref{divisorlem}.
\end{proof}

\subsection{The weighted Diamond-Halberstam-Richert sieve}
We now introduce a weighted sieve of Diamond, Halberstam and Richert, which we shall abbreviate as the weighted ``DHR" sieve. The DHR sieve is commonly used for higher dimensional sieve problems, and most of its theory is given in the monograph \cite{diamond2008higher}. In particular, the weighted sieve we apply is introduced in \cite[Chapter 11]{diamond2008higher} and is a variant of a weighted sieve due to Richert \cite{richert1969selberg}. We note that in some scenarios there may be slightly better sieve methods available (see e.g.\ \cite{franze2011sifting}). However, we have chosen to use the DHR sieve as it is simple to apply and has been comprehensively developed in the literature.

As in all sieve-theoretic problems, we wish to sieve some finite set of integers $\mathcal{A}$ by an infinite set of primes $\mathcal{P}$, such that no element of $\mathcal{A}$ is divisible by a prime in the complement $\overline{\mathcal{P}}=\{p\ \text{prime}:p\notin\mathcal{P}\}$. For the proof of Theorem \ref{mainthmtech}, we choose $\mathcal{A}$ and $\mathcal{P}$ to be
\begin{equation*}
    \mathcal{A}=\{N-p^k:p<N^{1/k},\:(p,N)=1\}\quad\text{and}\quad\mathcal{P}=\{p\ \text{prime}:(p,N)=1\}
\end{equation*}
or some variant thereof, depending on the value of $k$ and congruence class of $N$.

Further, for square-free integers $d$, let $X>1$ (independent of $d$) and $g(d)$ be a multiplicative function with $0\leq g(p)<p$ for all primes $p\in\mathcal{P}$. We then define
\begin{equation*}
    |\mathcal{A}_d|=\{a\in \mathcal{A}:d\mid a\}
\end{equation*}
and $r(d)$ to be such that
\begin{equation}\label{Adeqgen}
    |\mathcal{A}_d|=\frac{g(d)}{d}X+r(d),\qquad (d,\overline{\mathcal{P}})=1,
\end{equation}
where by $(d,\overline{\mathcal{P}})=1$ we mean that $(d,p)=1$ for all $p\in\overline{\PP}$.

To apply the weighted DHR sieve, we need four conditions to be satisfied. These can be found in \cite[Chapter 11.1]{diamond2008higher} but we state them here (with slightly different notation) for clarity. 

Firstly, we require a standard condition on $g(p)$ denoted $\mathbf{\Omega}(\kappa,L)$:
\begin{equation}\tag{$\mathbf{\Omega}(\kappa,L)$}\label{omegacon}
    \prod_{\substack{z_1\leq p<z_2\\p\in\mathcal{P}}}\left(1-\frac{ g(p)}{p}\right)^{-1}<\left(\frac{\log z_2}{\log z_1}\right)^{\kappa}\left\{1+\frac{L}{\log z_1}\right\},\quad \text{for}\ z_2>z_1\geq 2.
\end{equation}
Here, $L>0$ and $\kappa$ is called the \emph{dimension} of the sieve.

The next condition is an averaged bound for the remainder term $r(d)$. Namely, for some $B\geq 1$, we require
\begin{equation}\tag{$\mathbf{R_0(\kappa,\tau)}$}\label{r0con}
    \sum_{\substack{d<X^{\tau}/(\log X)^{B}\\ (d,\overline{\mathcal{P}})=1}}\mu^2(d)4^{\omega(d)}|r(d)|\ll\frac{X}{(\log X)^{\kappa+1}},
\end{equation}
where $\mu$ is the M\"obius function, $\omega(d)$ is the number of distinct prime divisors of $d$, the \emph{level of distribution} $\tau$ is preferably as large as possible, and $\kappa$ is as in \eqref{omegacon}.

For the third condition, we require that $\mathcal{A}$ does not have too many elements divisible by a square:
\begin{equation}\tag{$\mathbf{Q_0}$}\label{q0con}
    \sum_{\substack{z\leq q<y\\q\in\mathcal{P}}}|\mathcal{A}_{q^2}|\ll\frac{X\log X}{z}+y,\qquad 2\leq z<y.
\end{equation}
Finally, we need a measure of the largest element of $\mathcal{A}$:
\begin{equation}\tag{$\mathbf{M_0(\tau,\mu_0)}$}\label{m0con}
    \max_{a\in \mathcal{A}}|a|\leq X^{\tau\mu_0},
\end{equation}
with $\mu_0$ preferably as small as possible.

The weighted DHR sieve is then given as follows.

\begin{lemma}[{\cite[Theorem 11.1]{diamond2008higher}}]\label{DHRsieve}
    Given a sieve problem of dimension $\kappa\geq 1$ concerning a finite set of integers $\mathcal{A}$ and an infinite set of primes $\mathcal{P}$, suppose that the properties \eqref{omegacon}, \eqref{r0con}, \eqref{q0con} and \eqref{m0con} are satisfied. As defined in \cite[Theorem 6.1]{diamond2008higher}, let $f_{\kappa}$ and $F_{\kappa}$ denote the lower and upper sifting functions in the DHR sieve, and $\beta_{\kappa}$ denote the sifting limit. Let $r$ be a natural number satisfying 
    \begin{equation*}
        r>N(\kappa,\mu_0,\tau; u,v),
    \end{equation*}
    where 
    \begin{equation}\label{Ngenexp}
        N(\kappa,\mu_0,\tau;u,v):=\tau\mu_0u-1+\frac{\kappa}{f_{\kappa}(\tau v)}\int_{u}^vF_{\kappa}\left(v\left(\tau-\frac{1}{s}\right)\right)\left(1-\frac{u}{s}\right)\frac{\mathrm{d}s}{s},
    \end{equation}
    with $u$ and $v$ satisfying $\tau v>\beta_{\kappa}$ and $1/\tau<u<v$. Then
    \begin{equation}\label{repasym}
        |\{a\in\mathcal{A}:\Omega(a)\leq r,\: P^-(a)\geq X^{1/v}\}|\gg_L\  X\cdot \prod_{\substack{p<X^{1/v}\\ p\in\mathcal{P}}}\left(1-\frac{g(p)}{p}\right),\qquad X\to\infty
    \end{equation}
    where $P^-(a)$ denotes the smallest prime factor of $a$. In particular, if $X$ is sufficiently large and $L$ is independent of $X$, then $\mathcal{A}$ contains integers having at most $r$ prime factors (counting multiplicity), with each such factor greater than $X^{1/v}$.
\end{lemma}
After setting up the relevant sieve problem, the goal will be to compute \eqref{Ngenexp}. Although one can numerically compute \eqref{Ngenexp} for select values of $k$, here we use the following bound which makes our computations easy to reproduce and simple to extend to large $k$.

\begin{lemma}\label{Nklem}
    Let $N(\kappa,\mu_0,\tau,u,v)$ be as in Lemma \ref{DHRsieve} with $\kappa>1$. Let $\zeta$ and $\xi$ be parameters with $0<\zeta<\beta_{\kappa}$ and $\xi\geq\beta_{\kappa}$. Also let
     \begin{equation}\label{tauvtauu}
        \tau v=\xi+\frac{\xi}{\zeta}-1\qquad\text{and}\qquad\tau u=\zeta+1-\frac{\zeta}{\xi}.
    \end{equation}
    Then $N(\kappa,\mu_0,\tau;u,v)<\widetilde{N}(\kappa,\mu_0,\tau;\zeta,\xi)$, where
    \begin{align}\label{tildeNeq}
        &\widetilde{N}(\kappa,\mu_0,\tau;\zeta,\xi):=\\
        &\quad(1+\zeta)\mu_0-1-\kappa-(\mu_0-\kappa)\frac{\zeta}{\xi}+\left\{\kappa+\zeta\left(1-\frac{f_\kappa(\xi)}{f_\kappa(\xi+\xi/\zeta-1)}\right)\right\}\log\left(\frac{\xi}{\zeta}\right).\notag
    \end{align}
\end{lemma}
\begin{proof}
    This inequality is proven in \cite[Section 11.4]{diamond2008higher}. Note however that there is a small typo in \cite{diamond2008higher}, whereby the authors (erroneously) set $\tau v=\zeta+\xi/\zeta-1$ and thereby give a slightly different expression for $\widetilde{N}$.
\end{proof}

In order to apply Lemma \ref{Nklem} we require explicit values or bounds for $f_{\kappa}(\xi)$ and $\beta_{\kappa}$. For small $\kappa$, we achieve this by using Booker and Browning's supplementary code \cite{bookercode} and tables \cite{bookertable} for the paper \cite{booker2016square}. For larger values of $\kappa$, we instead use the following result.

\begin{lemma}\label{largebetalem}
    Let $\kappa\geq 2$ be a half-integer or integer, and $\beta_{\kappa}$ be the sifting limit of the DHR sieve as in Lemma \ref{DHRsieve}. Then
    \begin{equation}\label{betaupper}
        2\kappa<\beta_{\kappa}<3.75\kappa
    \end{equation}
    and
    \begin{equation}\label{fbeta}
        f_{\kappa}(\beta_{\kappa})=0.
    \end{equation}
\end{lemma}
\begin{proof}
    The bound \eqref{betaupper} follows from \cite[Proposition 17.3]{diamond2008higher} and \cite[Theorem 3.5]{diamond1993sieve}. Then, \eqref{fbeta} follows directly from the definition of $f_{\kappa}$ (see \cite[Theorem 6.1]{diamond2008higher}).
\end{proof}

\subsection{The Bombieri--Vinogradov theorem and \hbox{Elliott--Halberstam conjecture}}
Let
\begin{equation*}
    \pi(x)=\#\{p\ \text{prime}:p\leq x\}\quad\text{and}\quad\pi(x;d,a)=\#\{p\ \text{prime}:p\leq x,\: p\equiv a\ \text{(mod $d$)}\}
\end{equation*}
be the standard prime counting functions. As in many sieve-theoretic applications, we require the Bombieri-Vinogradov theorem to bound the average value of
\begin{equation*}
    \left|\pi(x;d,a)-\frac{\pi(x)}{\varphi(d)}\right|
\end{equation*}
with the goal of satisfying the sieve condition \eqref{r0con}. A standard statment of the Bombieri--Vinogradov theorem is as follows (see e.g.\ \cite[Theorem 9.18]{friedlander2010opera}).
\begin{lemma}[Bombieri--Vinogradov]\label{bomvinthm}
    For any $A>0$, one has 
    \begin{equation}\label{bomvineq}
        \sum_{d\leq D}\max_{(a,d)=1}\left|\pi(x;d,a)-\frac{\pi(x)}{\varphi(d)}\right|\ll_{h,A}\frac{x}{(\log x)^A},
    \end{equation}
    where $D=x^{\frac{1}{2}}/(\log x)^B$, for some constant $B>0$ depending on $A$. 
\end{lemma}

For the purposes of Theorems \ref{mainthmsimple} and \ref{mainthmtech}, we refer to the Elliott--Halberstam conjecture as the following improvement over Lemma \ref{bomvinthm}.

\begin{conjecture}[Elliott--Halberstam]\label{ellhalbcon}
    For any $A>0$ and $\varepsilon>0$, we have
    \begin{equation}\label{ellhalbeq}
        \sum_{d\leq D}\max_{(a,d)=1}\left|\pi(x;d,a)-\frac{\pi(x)}{\varphi(d)}\right|\ll_{\varepsilon,A}\frac{x}{(\log x)^A},
    \end{equation}
    where $D=x^{1-\varepsilon}$.
\end{conjecture}

In particular, Conjecture \ref{ellhalbcon} is a modern version of the Elliott--Halberstam conjecture, after Friedlander and Granville \cite{friedlander1989limitations} showed that it was impossible to take $D=x/(\log x)^B$ in \eqref{ellhalbeq}.

Now, with a view to verify the remainder condition \eqref{r0con} in the DHR sieve, we require the following (separate) consequences of the Bombieri--Vinogradov theorem and Elliott--Halberstam conjecture.
\begin{lemma}\label{bomvinthm2}
    Let $h$ be a positive integer and $A>0$. Then 
    \begin{equation}\label{bomvineq2}
        \sum_{d\leq D}\mu^2(d)h^{\omega(d)}\max_{(a,d)=1}\left|\pi(x;d,a)-\frac{\pi(x)}{\varphi(d)}\right|\ll_{A}\frac{x}{(\log x)^A},
    \end{equation}
    where $D=x^{\frac{1}{2}}/(\log x)^B$, for some constant $B>0$ depending on $A$ and $h$. Assuming the Elliott--Halberstam conjecture, we also have, for any $\varepsilon>0$
    \begin{equation}\label{ellhalbeq2}
        \sum_{d\leq D}\mu^2(d)h^{\omega(d)}\max_{(a,d)=1}\left|\pi(x;d,a)-\frac{\pi(x)}{\varphi(d)}\right|\ll_{\varepsilon,h,A}\frac{x}{(\log x)^A},
    \end{equation}
    where $D=x^{1-\varepsilon}$.
\end{lemma}

In particular, if we restrict to square-free $d$, then the sums in \eqref{bomvineq} and \eqref{ellhalbeq} can be weighted by a factor of $h^{\omega(d)}$ with essentially no change to the final bound. The proof of Lemma \ref{bomvinthm2} is standard, but we include the details below for completeness.

\begin{proof}[Proof of Lemma \ref{bomvinthm2}]
    We only provide a proof of \eqref{ellhalbeq2} since the proof of \eqref{bomvineq2} is essentially identical. So, assume the Elliott--Halberstam conjecture. We write
    \begin{equation*}
        E(d)=\max_{(a,d)=1}\left|\pi(x;d,a)-\frac{\pi(x)}{\varphi(d)}\right|\quad\text{and}\quad f(d)=\mu^2(d)h^{\omega(d)}.
    \end{equation*}
    By the Cauchy--Schwarz inequality and Elliott--Halberstam conjecture,
    \begin{equation}\label{cseq}
        \sum_{d\leq D}f(d)E(d)\leq\sqrt{\sum_{d\leq D}f(d)^2E(d)}\sqrt{\sum_{d\leq D}E(d)}\ll_A \sqrt{\sum_{d\leq D}f(d)^2E(d)}\sqrt{\frac{x}{(\log x)^A}}
    \end{equation}
    for any $A>0$. Since $D=x^{1-\varepsilon}$, an application of the Brun--Titchmarsh inequality yields $E(d)\ll_{\varepsilon}x/(\varphi(d)\log x)$. Thus, \eqref{cseq} can be bounded further as
    \begin{equation}\label{bruntitcheq}
         \sum_{d\leq D}f(d)E(d)\ll_{A,\varepsilon}\frac{x}{(\log x)^{(A+1)/2}}\sqrt{\sum_{d\leq D}\frac{f(d)^2}{\varphi(d)}}.
    \end{equation}
    Here,
    \begin{equation*}
        \sum_{d\leq D}\frac{f(d)^2}{\varphi(d)}=\sum_{d\leq D}\frac{\mu^2(d)h^{2\omega(d)}}{\varphi(d)}\leq\prod_{p\leq D}\left(1+\frac{h^2}{p-1}\right)\leq\prod_{p\leq D}\left(1+\frac{2h^2}{p}\right).
    \end{equation*}
    To finish, we note that
    \begin{align}\label{logheq}
        \prod_{p\leq D}\left(1+\frac{2h^2}{p}\right)=\exp\left(\sum_{p\leq D}\log\left(1+\frac{2h^2}{p}\right)\right)&\ll\exp\left(\sum_{p\leq D}\frac{2h^2}{p}+O_h(1)\right)\notag\\
        &\ll_h(\log x)^{2h^2},
    \end{align}
    where in the last step we used Mertens' first theorem. Substituting \eqref{logheq} into \eqref{bruntitcheq} gives the desired result (after a potential rescaling of $A$).
\end{proof}
    
\section{Proof of Theorem \ref{mainthmtech}}\label{techsect}
In this section we prove Theorem \ref{mainthmtech}. The general approach will be to apply the weighted DHR sieve, after carefully proving all of the required conditions. Along the way, we also apply Proposition \ref{thueprop} so that we may necessitate that $\eta$, appearing in \eqref{secondNeq}, is $k$th power-free.

\subsection{Setting up the sieve}\label{setupsect}
Throughout we fix $k\geq 1$. Let $\mathcal{Q}_k=\{q_1,q_2,\ldots,q_{\ell_k}\}$ be the set all primes such that $q_i-1\mid k$. Without loss of generality, we may assume that
\begin{equation*}
    2=q_1<q_2<\cdots<q_{\ell_k}.
\end{equation*}
For any positive integer $N$, let $m_i$ be as defined in Theorem \ref{mainthmtech}. Then, for each $i\in\{1,\ldots,\ell_k\}$, let $n_i$ be the largest possible value (amongst all $N\geq 1$) of $m_i$. That is, $n_i$ is the largest integer such that
\begin{align}
    q_i^{n_i-1}(q_i-1)&\mid k,\quad\text{if $q_i\neq 2$ or $n_i\leq 2$},\label{nicon1}\\
    2^{n_i-2}&\mid k,\quad \text{if $q_i=2$, $k$ is even and $n_i\geq 3$.}\label{nicon2}
\end{align}
For any integer $N\geq 1$, we then set
\begin{align}
    Q&=\prod_{1\leq i\leq \ell_k}q_i^{m_i},\notag\\
    \mathcal{P}&=\{p\ \text{prime}:p\notin\mathcal{Q}_k,\:(p,N)=1\}\label{mainPdef}
\end{align}
and
\begin{equation}\label{mainAdef}
    \mathcal{A}=\left\{\frac{N-p^k}{Q}: 2\leq p<N^{1/k},\: p\not\in\mathcal{Q}_k,\:(p,N)=1,\: C_{N,p,k}\right\},
\end{equation}
where $C_{N,p,k}$ is the set of $\ell_k$ conditions
\begin{align}
    N-p^k\not\equiv 0\pmod{q_i^{n_i+1}},\quad\text{if $m_i=n_i$}.\label{cnpkeq}
\end{align}
Our aim is to sift the set $\mathcal{A}$ by $\mathcal{P}$, allowing us to find $a\in \mathcal{A}$ with a bounded number of prime factors. As discussed in Section \ref{kthressect}, each element of $\mathcal{A}$ is a positive integer. However, we also need to ensure that no element of $\mathcal{A}$ is divisible by a prime in $\overline{\mathcal{P}}$.

\begin{lemma}
    Let $\mathcal{A}$ and $\mathcal{P}$ be as defined in \eqref{mainAdef} and \eqref{mainPdef} above. Then for all $a\in\mathcal{A}$ and $\widetilde{p}\in\overline{\mathcal{P}}$, we have $\widetilde{p}\nmid a$.
\end{lemma}
\begin{proof}
    To begin with, we note that
    \begin{equation*}
        \overline{\PP}=\{p\ \text{prime}:p\in\mathcal{Q}_k\ \text{or}\ p\mid N\}.
    \end{equation*}
    Now, let $a=(N-p^k)/Q$ be an element of $\mathcal{A}$, and $\widetilde{p}\in\overline{\mathcal{P}}$. First, suppose that $\widetilde{p}\mid N$. Then, since $(p,N)=1$ it follows that $(a,N)=1$ and therefore $\widetilde{p}\nmid a$.

    On the other hand suppose that $\widetilde{p}\in\mathcal{Q}_k$. That is, $\widetilde{p}=q_i$ for some $i\in\{1,\ldots,\ell_k\}$. We consider two cases: $m_i\neq n_i$ and $m_i=n_i$. If $m_i\neq n_i$ then $N\equiv 1\pmod{q_i^{m_i}}$ but $N\not\equiv 1\pmod{q_i^{m_i+1}}$. Additionally, $p^k\equiv 1\pmod{q_i^{m_i}}$ by Proposition \ref{Qprop}, so it follows that $\widetilde{p}=q_i\nmid a$ as required. Otherwise, if $m_i=n_i$ the result follows analogously by the condition \eqref{cnpkeq}.
\end{proof}

\subsection{The \eqref{omegacon} condition}
Having defined $\mathcal{A}$ and $\mathcal{P}$ for our problem, we now define a suitable multiplicative function $g(d)$ and prove the first condition \eqref{omegacon} for the weighted DHR sieve.

In particular, for any $d$ with $(d,\overline{\mathcal{P}})=1$, we let
\begin{equation}\label{rhokdef}
    \rho_k(d):=\prod_{p\mid d}\gcd(k, p-1)
\end{equation}
and set, for square-free $d>0$,
\begin{equation}\label{gdeqproof}
    g(d) = 
    \begin{cases}
        \rho_k(d)\frac{d}{\varphi(d)}, & \text{if $(d,\overline{\mathcal{P}})=1$ and $N$ is a $k$th power residue mod $d$},\\
        0,& \text{otherwise}.
    \end{cases}
\end{equation}

Now, $2\in\mathcal{Q}_k$ for all $k\geq 1$, so it follows that any $d$ with $(d,\overline{\mathcal{P}})=1$ is odd. Thus, by Lemma \ref{localoblem} and the Chinese remainder theorem, $\rho_k(d)$ counts the number of solutions to the congruence $x^k\equiv N\pmod{d}$. Note that a corresponding expression for $X$ and consequently $r(d)$ in \eqref{Adeqgen} will be given in the next subsection (see Proposition \ref{rdprop}).

With $g(d)$ defined as such, we now show that \eqref{omegacon} is satisfied for $\kappa=d(k)$ and an (unspecified) constant $L$ depending on $k$.  

\begin{proposition}\label{omegaprop}
    Let $k\geq 1$, $N\geq 1$, and $\mathcal{P}$ and $g(d)$ be as in \eqref{mainPdef} and \eqref{gdeqproof} respectively. Then,
    \begin{equation*}      
        \prod_{\substack{z_1\leq p<z_2\\p\in\mathcal{P}}}\left(1-\frac{ g(p)}{p}\right)^{-1}<\left(\frac{\log z_2}{\log z_1}\right)^{d(k)}\left\{1+O_k\left(\frac{1}{\log z_1}\right)\right\},\qquad z_2>z_1\geq 2,
    \end{equation*}
    where $d(k)$ is the number of divisors of $k$.
\end{proposition}

To prove Proposition \ref{omegaprop} we first state a couple of lemmas from the literature.
\begin{lemma}[Mertens' second theorem for primes in arithmetic progression]\label{mertenaplem}
    We have
    \begin{equation*}
        \sum_{\substack{p\leq x\\p\equiv a\ (\mathrm{mod}\ k)}}\frac{1}{p}=\frac{\log\log x}{\varphi(k)}+M_{k,a}+O_k\left(\frac{1}{\log x}\right),
    \end{equation*}
    where $M_{k,a}$ is constant depending only on $k$ and $a$.
\end{lemma}
\begin{proof}
    See, for example \cite{languasco2010computing}.
\end{proof}
\begin{lemma}[Menon's identity]\label{menonlem}
    For every integer $k\geq 1$
    \begin{equation*}
        \sum_{\substack{1\leq a\leq k\\(a,k)=1}}\gcd(a-1,k)=\varphi(k)d(k).
    \end{equation*}
\end{lemma}
\begin{proof}
    See \cite{menon1965sum} or \cite{toth2021proofs}.
\end{proof}
We now prove Proposition \ref{omegaprop}.
\begin{proof}[Proof of Proposition \ref{omegaprop}]
    For any prime $p\geq 2$, we define
    \begin{equation}\label{gtildedef}
        \widetilde{g}(p):=
        \begin{cases}
            \gcd(k,p-1)\frac{p}{p-1},&\text{if $p-1\nmid k$},\\
            0,&\text{if $p-1\mid k$}.
        \end{cases}
    \end{equation}
    In particular, $g(p)\leq\widetilde{g}(p)$ for all $p\in\mathcal{P}$, and for $p-1>k$ we have $0\leq\widetilde{g}(p)<p$. As we are interested in an asymptotic result as $z_1\to\infty$, it thus suffices to bound, for sufficiently large $z_1$,
    \begin{equation*}
        \prod_{z_1\leq p<z_2}\left(1-\frac{\widetilde{g}(p)}{p}\right)^{-1}
    \end{equation*}
    which is independent of $N$. Since
    \begin{equation*}
        \frac{1}{p-1}=\frac{1}{p}+\frac{1}{p(p-1)},
    \end{equation*}
    it follows that for $x\geq k$
    \begin{align*}
        \sum_{p\leq x}\frac{\widetilde{g}(p)}{p}&= \sum_{\substack{p\leq x\\p-1\nmid k\\(p,k)=1}}\gcd(p-1,k)\left(\frac{1}{p}+\frac{1}{p(p-1)}\right)+\sum_{\substack{p\leq x\\p-1\nmid k\\p\mid k}}\frac{\gcd(p-1,k)}{p-1}\\
        &=\sum_{\substack{p\leq x\\(p,k)=1}}\frac{\gcd(p-1,k)}{p}+A_k
    \end{align*}
    for some constant $A_k$ depending only on $k$. Thus, by Lemmas \ref{mertenaplem} and \ref{menonlem},
    \begin{equation*}
        \sum_{p\leq x}\frac{\widetilde{g}(p)}{p}=d(k)\log\log x+B_k+O_k\left(\frac{1}{\log x}\right)
    \end{equation*}
    for some constant $B_k$ depending only on $k$. Therefore
    \begin{equation}\label{finalgpbound}
        \sum_{z_1\leq p< z_2}\frac{\widetilde{g}(p)}{p}=d(k)\log\frac{\log z_2}{\log z_1}+O_k\left(\frac{1}{\log z_1}\right).
    \end{equation}
    Next we note that $\widetilde{g}(p)\leq p/2$ for $p\geq 2k+1$. Thus, for $z_1\geq 2k+1$, 
    \begin{align}
        \prod_{z_1\leq p<z_2}\left(1-\frac{\widetilde{g}(p)}{p}\right)^{-1}&=\exp\left(-\sum_{z_1\leq p<z_2}\log\left(1-\frac{\widetilde{g}(p)}{p}\right)\right)\notag\\
        &<\exp\left(\sum_{z_1\leq p<z_2}\left(\frac{\widetilde{g}(p)}{p}+\frac{\widetilde{g}(p)^2}{p^2}\right)\right)\label{gpgp2eq}
    \end{align}
    since $\log(1-x)>-x-x^2$ for all $0<x\leq 1/2$. From the definition \eqref{gtildedef} of $\widetilde{g}(p)$, we have $\widetilde{g}(p)\leq 2k$ so that $\widetilde{g}(p)^2\leq 4k^2=O_k(1)$ as $z_1\to\infty$. Using this fact, along with \eqref{finalgpbound}, we can bound \eqref{gpgp2eq} above by
    \begin{align*}
        \exp\left(d(k)\log\frac{\log z_2}{\log z_1}+O_k\left(\frac{1}{\log z_1}\right)+4k^2\sum_{n\geq z_1}\frac{1}{n^2}\right)&=\left(\frac{\log z_2}{\log z_1}\right)^{d(k)}\exp\left(O_k\left(\frac{1}{\log z_1}\right)\right)\\
        &=\left(\frac{\log z_2}{\log z_1}\right)^{d(k)}\left\{1+O_k\left(\frac{1}{\log z_1}\right)\right\},
    \end{align*}
    as required.
\end{proof}

\subsection{The \ref{r0con} condition}
Next we prove the condition \ref{r0con} with $\tau=1/2$. To begin with, we require the following result which is an application of Lemma \ref{localoblem} and Proposition \ref{Qprop}.

\begin{lemma}\label{xk1blem}
    Let $k\geq 1$ and let $q_i$ and $n_i$ be as in Section \ref{setupsect} with $i\in\{1,\ldots,\ell_k\}$. Then, for every integer $x$ with $(x,q_i)=1$, there exists an integer $0\leq b<q_i$ with
    \begin{equation}\label{xkbeq}
        x^k\equiv 1+bq_i^{n_i}\pmod{q_i^{n_i+1}}.
    \end{equation}
    In addition, for each $0\leq b<q_i$, the congruence \eqref{xkbeq} has exactly $q_i^{n_i-1}(q_i-1)$ solutions for $x$ mod $q_i^{n_i+1}$ with $(x,q_i)=1$.
\end{lemma}
\begin{proof}
    For the first part of the lemma, we note that by the definition of $n_i$ and Proposition \ref{Qprop},
    \begin{equation*}
        x^k\equiv 1\pmod{q_i^{n_i}}
    \end{equation*}
    for all integers $x$ with $(x,q_i)=1$. Hence, \eqref{xkbeq} holds for some $b$ with $0\leq b<q_i$. 
    
    For the second part of the lemma, we split into the case where $q_i$ is odd, and the case where $q_i=2$. If $q_i$ is odd, and $1+bq_i^{n_i}$ is a $k$th power residue, then by Lemma \ref{localoblem} the congruence \eqref{xkbeq} has
    \begin{equation*}
        \gcd(k,\varphi(q_i^{n_i+1}))=\gcd(k,q_i^{n_i}(q_i-1))=q_i^{n_i-1}(q_i-1)
    \end{equation*}
    solutions with $(x,q_i)=1$, noting that
    \begin{equation*}
        q_i^{n_i-1}(q_i-1)\mid k\quad\text{but}\quad q_i^{n_i}(q_i-1)\nmid k
    \end{equation*}
    by \eqref{nicon1}. However, by the generalised pigeonhole principle this must hold for all $q_i$ choices of $b$ (meaning $1+bq_i^{n_i}$ is always a $k$th power residue) since there are
    \begin{equation*}
        \varphi(q_i^{n_i+1})=q_i\cdot q_i^{n_i-1}(q_i-1)
    \end{equation*}
    possible values of $x$ coprime to $q_i^{n_i+1}$. The case $q_i=2$ follows analogously, using \eqref{nicon1} if $k$ is odd (and $n_i=1$), or \eqref{nicon2} if $k$ is even (and $n_i\geq 3$).
\end{proof}

We now prove \eqref{r0con} for our sifting problem. This essentially boils down to a careful application of the Bombieri--Vinogradov theorem.

\begin{proposition}\label{rdprop}
    Let $k\geq 1$, $N\geq 1$, and $q_i$, $m_i$, $n_i$, $Q$, $\mathcal{P}$ and $\mathcal{A}$ be as in Section~\ref{setupsect}. Let
    \begin{equation*}
        \epsilon_i(N)=
        \begin{cases}
            1, & \text{if}\ m_i = n_i,\\
            0,& \text{if}\ m_i<n_i
        \end{cases}
    \end{equation*}
    and set
    \begin{equation}\label{mainXeq}
        X=\frac{\alpha_k(N)\gamma_k(N)\pi(N^{1/k})}{\varphi\left(Q\cdot\prod_i q_i^{\epsilon_i(N)}\right)}
    \end{equation}
    where
    \begin{equation*}
        \alpha_k(N)=\prod_{\substack{1\le i \le \ell_k\\\epsilon_i(N)=1}}q_i^{n_i-1}(q_i-1)^2,
    \end{equation*}
    and 
    \begin{equation*}
        \gamma_k(N)=
        \begin{cases}
            \prod_{i}\gcd\left(k, \varphi\left(q_i^{m_i(1-\epsilon_i(N))}\right)\right), &\text{if}\ m_1(1-\epsilon_1(N))\le 2,\\
            2\cdot \gcd(k, 2^{m_1-2})\left(\prod_{i>1}\gcd\left(k, \varphi\left(q_i^{m_i(1-\epsilon_i(N))}\right)\right)\right), &\text{if}\ m_1(1-\epsilon_1(N))\ge 3.
        \end{cases}
    \end{equation*}
    With $g(d)$ as in \eqref{gdeqproof} and $r(d)=|\mathcal{A}_d|-\frac{g(d)}{d}X$, we then have, for any $A>1$,
    \begin{equation}\label{rdproofeq}
        \sum_{\substack{d<X^{1/2}/(\log X)^{B}\\ (d,\overline{\mathcal{P}})=1}}\mu^2(d)4^{\omega(d)}|r(d)|\ll_{k,A}\frac{X}{(\log X)^{A}},
    \end{equation}
    where $B>0$ depends on $A$.
\end{proposition}
\begin{proof}
    To begin with, since the number of primes dividing $N$ is $\ll \log N$, we have    \begin{equation}\label{Admaindef}
        |\mathcal{A}_d|=\left|\left\{2\leq p<N^{1/k}: p^k\equiv N\ (\text{mod}\ Qd),\: C_{N,p,k}\right\}\right|+O(\log N),
    \end{equation}
    for any square-free $d$ with $(d,\overline{\mathcal{P}})=1$.
    The congruence classes of primes belonging to the set on the right-hand side of \eqref{Admaindef} correspond to the solutions of the system
    \begin{align*}
        &N-x^k\not\equiv 0\pmod{q_i^{n_i+1}},\quad\text{if}\ m_i=n_i,\\
        &N-x^k\equiv 0 \pmod{Qd}
    \end{align*}
    with $(x,q_1\cdots q_{\ell_k})=1$. In order to apply the Chinese remainder theorem, we recast this system as
    \begin{align}
        N-x^k&\not\equiv 0\pmod{q_i^{n_i+1}},\quad\text{if}\ m_i=n_i,\label{kcong1}\\
        N-x^k&\equiv 0 \quad \big(\mathrm{mod}\ Q/\prod_{i}q_i^{m_i\epsilon_i(N)}\big),\label{kcong2}\\
        N-x^k&\equiv 0\pmod{d}\label{kcong3}.
    \end{align}
    For each $i\in\{1,\ldots,\ell_k\}$ with $m_i=n_i$, we have $N\equiv 1\pmod{q_i^{n_i}}$ so that the number of solutions to \eqref{kcong1} is $\alpha_k(N)$ by Lemma \ref{xk1blem}. Then, by Lemma \ref{localoblem}, the number of solutions to \eqref{kcong2} is $\gamma_k(N)$, and the number of solutions to \eqref{kcong3} is $\rho_k(d)$ (as defined in \eqref{rhokdef}). Thus, by the Chinese remainder theorem, the system \eqref{kcong1}--\eqref{kcong3} has $\alpha_k(N)\gamma_k(N)\rho_k(d)$ solutions mod $Qd\cdot\prod_i q_i^{\epsilon_i(N)}$. 

    We now let $\delta_1,\delta_2,\ldots,\delta_{\alpha_k(N)\gamma_k(N)\rho_k(d)}$ denote all the solutions to \eqref{kcong1}--\eqref{kcong3} so that \eqref{Admaindef} can be rewritten as
    \begin{align*}
        |A_d|=\sum_{j=1}^{ \alpha_{k}(N)\gamma_k(N)\rho_k(d)}&\left|\left\{2\leq p<N^{1/k}:p\equiv \delta_j\ \big(\mathrm{mod}\ Qd\cdot \prod_i q_i^{\epsilon_i(N)}\big)\right\}\right|+O(\log N).
    \end{align*}
    Since
    \begin{equation*}
        \frac{g(d)}{d}X=\alpha_k(N)\gamma_k(N)\rho_k(d)\frac{\pi(N^{1/k})}{\varphi\left(Qd\cdot\prod_i q_i^{\epsilon_i(N)}\right)}
    \end{equation*}
    for any $d$ with $(d,\overline{\mathcal{P}})=1$, we then have, for $D\geq 1$,
    \begin{align}
        &\sum_{\substack{d<D\\ (d,\overline{\mathcal{P}})=1}}\mu^2(d)4^{\omega(d)}|r(d)|\notag\\
        &=\sum_{\substack{d<D\\ (d,\overline{\mathcal{P}})=1}}\mu^2(d)4^{\omega(d)}\left|\sum_{j=1}^{ \alpha_{k}(N)\gamma_k(N)\rho_k(d)}\pi(N^{1/k};Qd\prod_{i}q_i^{\epsilon_i(N)},\delta_j)-\frac{g(d)}{d}X\right|+O(\log N)\notag\\
        &\ll_k\sum_{\substack{d<D\\ (d,\overline{\mathcal{P}})=1}}\mu^2(d)\cdot (4k)^{\omega(d)}\left|\pi(N^{1/k};Qd\prod_{i}q_i^{\epsilon_i(N)},\delta_j)-\frac{\pi(N^{1/k})}{\varphi\left(Qd\cdot\prod_i q_i^{\epsilon_i(N)}\right)}\right|+\log N,\label{rdfinaleq}
    \end{align}
    where we have used that $\alpha_k(N)\gamma_k(N)=O_k(1)$ and $\rho_k(d)\leq k^{\omega(d)}$. The desired result \eqref{rdproofeq} then follows by applying the Bombieri--Vinogradov theorem (or more precisely, Lemma~\ref{bomvinthm2}) to \eqref{rdfinaleq}.
\end{proof}

\subsection{The \eqref{q0con} and \eqref{m0con} conditions}
We now prove that the conditions \eqref{q0con} and \eqref{m0con} are satisfied for our sifting problem. This is rather straightforward compared to our verifications of \eqref{omegacon} and \eqref{r0con} in the preceding sections.

\begin{lemma}\label{q0lem}
    Let $\mathcal{P}$ and $\mathcal{A}$ be as defined in \eqref{mainPdef} and \eqref{mainAdef}, and $X$ be as in \eqref{mainXeq}. Then, the condition \eqref{q0con} holds.
\end{lemma}
\begin{proof}
    For any prime $q\in\PP$, we have
    \begin{equation*}
        |\mathcal{A}_{q^2}|\leq \left|\left\{2\leq p\leq N^{1/k}:p^k\equiv N\ (\text{mod}\ q^2)\right\}\right|.
    \end{equation*}
    Provided $q(q-1)>k$, Lemma \ref{localoblem} gives that the congruence $x^k\equiv N$ (mod $q^2$) has at most $k$ solutions mod $q^2$. Hence, for such $q$,
    \begin{equation*}
        |\mathcal{A}_{q^2}|\leq k\left(\frac{N^{1/k}}{q^2}+1\right).
    \end{equation*}
    Thus,
    \begin{equation*}
        \sum_{\substack{z\leq q<y\\q\in\mathcal{P}}}|\mathcal{A}_{q^2}|\ll_k\frac{N^{1/k}}{z}+y\ll_k\frac{X\log X}{z}+y,
    \end{equation*}
    as desired. Here, $X\asymp_k N^{1/k}/\log N$ by the prime number theorem.
\end{proof}

\begin{lemma}\label{m0lem}
    Let $\mathcal{A}$ and $X$ be as defined in \eqref{mainAdef} and \eqref{mainXeq} respectively. Then, for any $\varepsilon>0$ and $N\geq N_{\varepsilon}$ sufficiently large, the condition \eqref{m0con} holds for $\tau=1/2$ and $\mu_0=2k+\varepsilon$.
\end{lemma}
\begin{proof}
    The proof follows since all $a\in\mathcal{A}$ satisfy $a\leq N$, and $X\asymp_k N^{1/k}/\log N$ by the prime number theorem.
\end{proof}

\subsection{Necessitating that $\eta$ is $k$th power-free}
Before applying the DHR sieve, we bound the number of $a\in\mathcal{A}$ that are divisible by a $k$th power after small prime factors have been sieved out. This will allow us to necessitate that $\eta$ in Theorem \ref{mainthmtech} is $k$th power-free. Similar to Erd\H{o}s and Rao's proof of Theorem \ref{erdraothm}, we use Proposition \ref{thueprop} to bound the number of solutions to an equation involving a binary form. However, note that one can not adapt Erd\H{o}s and Rao's proof directly, since we are considering a sparser set of $a\in\mathcal{A}$. In particular, we have to make careful use of the fact that the sieving process removes small prime factors in $\mathcal{A}$. 

\begin{proposition}\label{akvprop}
    Let $k\geq 2$ with $\mathcal{A}$ and $X$ as defined in \eqref{mainAdef} and \eqref{mainXeq} respectively. For $v\geq 2$ define
    \begin{equation*}
        \mathcal{S}_{k,v}=\{a\in\mathcal{A}:P^-(a)\geq X^{1/v},\:\exists q\in\PP\ \text{such that}\ q^k\mid a\}
    \end{equation*}
    to be the set of $a\in\mathcal{A}$ with prime factors larger than $X^{1/v}$ and divisible by a $k$th power. Then,
    \begin{equation}\label{Akvbound}
        |\mathcal{S}_{k,v}|\ll_{k,v} X^{1-\frac{1}{2kv}}.
    \end{equation}
\end{proposition}
\begin{remark}
    With more care, the exponent $1-1/2kv$ in \eqref{Akvbound} could be lowered. However, for our purposes any exponent strictly less than 1 is sufficient.
\end{remark}
\begin{proof}
    We have
    \begin{equation}\label{Akvineq}
        |\mathcal{S}_{k,v}|=\sum_{\substack{X^{1/v}<q<N^{1/k}\\ q\in\PP}}\sum_{\substack{a\in\mathcal{A}\\P^{-}(a)\geq X^{1/v}\\q^k\mid a}}1\leq \sum_{\substack{X^{1/v}<q<N^{1/k}\\ q\in\PP}}\sum_{\substack{p^k\equiv N\ \text{(mod}\ q^k\text{)}\\P^{-}(N-p^k)\geq X^{1/v}\\(p,N)=1}}1.
    \end{equation}
    We split the right-hand side of \eqref{Akvineq} into $S_1+S_2$, where
    \begin{align*}
        S_1=\sum_{\substack{X^{1/v}\leq q<N^{1/k}X^{-1/kv}\\ q\in\PP}}\:\sum_{\substack{p^k\equiv N\ \text{(mod}\ q^k\text{)}\\P^{-}(N-p^k)\geq X^{1/v}\\ (p,N)=1}}1\\
        S_2=\sum_{\substack{N^{1/k}X^{-1/kv}\leq q< N^{1/k}\\ q\in\PP}}\:\sum_{\substack{p^k\equiv N\ \text{(mod}\ q^k\text{)}\\P^{-}(N-p^k)\geq X^{1/v}\\ (p,N)=1}}1
    \end{align*}
    The congruence $p^k\equiv N\pmod{q^2}$ has at most $k$ solutions mod $q^2$. Thus, noting that $N^{1/k}\asymp_k X\log X$ by the prime number theorem,
    \begin{align*}
        S_1&\leq\sum_{X^{1/v}<q<N^{1/k}X^{-1/kv}}k\left(\frac{N^{1/k}}{q^k}+1\right)\\
        &\ll_{k,v} \frac{N^{1/k}}{X^{(k-1)/kv}}+X^{1-\frac{1}{kv}}\log X\\
        &\ll_{k,v} X^{1-\frac{(k-1)}{kv}}\log X+X^{1-\frac{1}{kv}}\log X\ll_{k,v} X^{1-\frac{1}{2kv}}.
    \end{align*}
    We now bound $S_2$. Here, $q\geq N^{1/k}X^{-1/v}$ and we are searching for solutions to
    \begin{equation*}
        p^k\equiv N\ \text{(mod}\ q^k\text{)}\implies N=p^k+mq^k
    \end{equation*}
    with $m\geq 1$ and $P^-(mq^k)>X^{1/v}$. We claim that only $m=1$ is possible. Suppose otherwise. Then $m\geq X^{1/v}$ and
    \begin{equation*}
        N-p^k=mq^k\geq X^{1/v}\left(N^{1/k}X^{-1/kv}\right)^k=N,
    \end{equation*}
    a contradiction. Hence, by Proposition \ref{thueprop}, we have for any $\varepsilon>0$
    \begin{equation*}
        S_2\leq\sum_{q<N^{1/k}}\:\sum_{\substack{(p,N)=(q,N)=1\\ N=p^k+q^k}}1\ll_k N^{\varepsilon}.
    \end{equation*}
    Setting $\varepsilon$ small enough so that $N^{\varepsilon}\ll X^{1-1/2kv}$ (e.g.\ $\varepsilon=1/2k$) finishes the proof.
\end{proof}

\subsection{Applying the DHR weighted sieve}
With all the conditions proved in the preceding sections, we now directly apply the DHR sieve and prove Theorem \ref{mainthmtech}.

\begin{proof}[Proof of Theorem \ref{mainthmtech}]
    To begin with, the case $k=1$ is covered by Chen and Li's theorems (Theorems \ref{chenthm} and \ref{lithm}) so it suffices to set $k\geq 2$.

    Let $\mathcal{A}$, $\mathcal{P}$ and $g(d)$ be as defined in \eqref{mainAdef}, \eqref{mainPdef} and \eqref{gdeqproof} respectively. By Propositions \ref{omegaprop} and \ref{rdprop}, and Lemmas \ref{q0lem} and \ref{m0lem}, we have that \eqref{omegacon}, \eqref{r0con}, \eqref{q0con} and \eqref{m0con} are satisfied with 
    \begin{equation}\label{kappataumu}
        \kappa=d(k),\ \tau=1/2, \ \text{and}\ \mu_0=2k+\varepsilon 
    \end{equation}
    for any choice of $\varepsilon>0$ and sufficiently large $N\geq N_{\varepsilon}$. Hence, by Lemmas \ref{DHRsieve} and \ref{Nklem}, we can take
    \begin{equation*}
        \widetilde{M}(k)=\left\lceil\widetilde{N}(\zeta,\xi)\right\rceil,
    \end{equation*}
    with $\widetilde{N}(\zeta,\xi)=\widetilde{N}(\kappa,\mu_0,\tau;\zeta,\xi)$ and the parameters $\zeta$ and $\xi$ as in Lemma \ref{Nklem}. More precisely, let
    \begin{equation*}
        \mathcal{A}_{k,v}=\{a\in\mathcal{A}:\Omega(a)\leq \widetilde{M}(k),\: P^-(a)\geq X^{1/v}\}
    \end{equation*}
    with $v=v(\zeta,\xi)$ as defined in \eqref{tauvtauu}. By \eqref{repasym} and Proposition \ref{omegaprop},
    \begin{align}
        |\mathcal{A}_{k,v}|\gg\  X\cdot \prod_{\substack{p<X^{1/v}\\ p\in\mathcal{P}}}\left(1-\frac{g(p)}{p}\right)\gg_{k,v}\frac{X}{(\log X)^{d(k)}}.\label{Xdkbound}
    \end{align}
    Then, by Proposition \ref{akvprop}, the number of $a\in\mathcal{A}_{k,v}$ divisible by a $k$th power is of smaller order than \eqref{Xdkbound} as $X\to\infty$. That is, for any $k\geq 2$, there are infinitely many $a\in\mathcal{A}$ which are $k$th power-free and have at most $\widetilde{M}(k)$ prime factors.
    
    We now describe our computations used to find suitable values of $\zeta$ and $\xi$ that minimise $\widetilde{N}$ and thus $\widetilde{M}(k)$. For further details, the code and tables are available on Github \cite{djsimoncode}.\\
    
    \textbf{Small values of $k$ ($2\leq k\leq 10)$:}
    We start with the case $k=2$ which corresponds to $\kappa=2$, $\tau=1/2$ and $\mu_0=4+\varepsilon$. In this setting, 
    \begin{align}\label{tildeNeq2}
        &\widetilde{N}(\zeta,\xi)=\\
        &\quad (1+\zeta)(4+\varepsilon)-3-(2+\varepsilon)\frac{\zeta}{\xi}+\left\{2+\zeta\left(1-\frac{f_\kappa(\xi)}{f_\kappa(\xi+\xi/\zeta-1)}\right)\right\}\log\left(\frac{\xi}{\zeta}\right)\notag,
    \end{align}
    where, for the purposes of computation, we set $\varepsilon=10^{-10}$. 
    
    The main computational difficulty comes from computing values of $f_{\kappa}(\xi)$ for $\xi\geq\beta_{2}$. Since $f_2(\xi)$ does not have a simple expression, we compute accurate upper bounds, denoted by $f^+_2(\xi)$, and lower bounds, denoted by $f_2^-(\xi)$. This way, an upper bound for $\widetilde{N}$ can be computed for any $\zeta$ and $\xi$ by noting that
    \begin{equation*}
        \frac{f_{2}(\xi)}{f_{2}(\xi+\xi/\zeta-1)}\geq\frac{f^{-}_{2}(\xi)}{f^{+}_{2}(\xi+\xi/\zeta-1)}.
    \end{equation*}
    In this framework, we begin by noting that the sifting limit $\beta_2=4.266\ldots$ is given to 20 decimal places in \cite{booker2016square}. Then, for $\xi\geq \beta_2$, we compute values for $f_2^+(\xi)$ and $f_2^-(\xi)$ using Booker and Browning's code \cite{bookercode}. In particular, for
    \begin{equation}\label{xiset}
        \xi\in\{4.27,4.28,\ldots, 10.33, 10.34\}
    \end{equation}
    we computed $f_2(\xi)$, setting $f_2(\xi)^{+}$ to be the value of $f_2(\xi)$ rounded up to 5 decimal places. Since $f_2^+(10.34)=1$ and $f_2$ increases monotonically towards 1 \cite[Theorem 6.1]{diamond2008higher}, we set $f_2^+(\xi)=1$ for all $\xi\geq 10.34$. A lower bound for $f_2(\xi)$ is then obtained by setting
    \begin{equation*}
        f_2^-(\xi)=f_2^+(\xi)-10^{-5}.
    \end{equation*}
    A final technical case to consider is when $4.27<\xi<10.34$ is not in the set \eqref{xiset}. However, this can be handled by again using the fact that $f_{\kappa}$ is increasing. For instance, if $4.27<\xi<4.28$ then
    \begin{equation*}
        f_2^{-}(4.27)<f_2(\xi)<f_2^+(4.28) 
    \end{equation*}
    so that it is valid to set $f_2^-(\xi)=f_2^-(4.27)$ and $f_2^+(\xi)=f_2^+(4.28)$.
    
    Equipped with these lower and upper bounds for $f_2$, we can compute an upper bound for $\widetilde{N}$ in \eqref{tildeNeq2}. Here, for ease of computation we restrict to $\zeta\geq 0.01$ and $\xi\geq 4.27$. To find the best values of $\zeta$ and $\xi$ we use the \emph{Minimize} function in \textit{Mathematica}$^{\text{\sffamily\textregistered}}$. To 2 decimal places, this yields $\zeta=0.49$ and $\xi=9.39$ which gives $\widetilde{N}(\kappa,\mu_0,\tau;\zeta,\xi)<7.88$ so that $\widetilde{M}(2)=8$ works.

    Repeating the same procedure for $3\leq k\leq 10$ yields the values of $\widetilde{M}(k)$ in Table~\ref{Mktable}. The corresponding values of $\zeta$, $\xi$ and $\widetilde{N}$ are then provided in Table \ref{table:params}.

    \def\arraystretch{1.4}
    \begin{table}[h]
    \centering
    \caption{Corresponding values of $\zeta$, $\xi$, $\widetilde{N}$ and $\widetilde{M}(k)$ for $2~\leq~k\leq~10$.}
    \label{table:params}
    \begin{tabular}{|c|c c c c|} 
    \hline
    $k$ & $\zeta$ & $\xi$ & $\widetilde{N}(\kappa,\mu_0,\tau;\zeta,\xi)$ & $\widetilde{M}(k)$ \\ 
    \hline
        2 & 0.49 & 5.25 & 7.88 & 8 \\
        3& 0.34 & 4.93 & 10.54 & 11 \\
        4& 0.37 & 7.57 & 16.24 & 17 \\
        5& 0.20 & 4.50 & 15.35 & 16 \\
        6& 0.32 & 9.97 & 24.87 & 25 \\
        7& 0.15 & 4.28 & 19.88 & 20 \\
        8& 0.24 & 9.66 & 29.88 & 30 \\
        9& 0.16 & 6.75 & 28.30 & 29 \\
        10& 0.19 & 9.38 & 34.65 & 35 \\ [1ex] 
    \hline
    \end{tabular}
    \end{table}

    \textbf{Large values of $k$ ($k>10$):} We now focus on $k>10$, with the aim of showing that $\widetilde{M}(k)=4k+1$ is valid for all even $k$ and $\widetilde{M}(k)=4k-1$ is valid for all odd $k$. To simplify the expression for $\widetilde{N}$, we set $\xi=\beta_{d(k)}$ so that $f_{d(k)}(\xi)=0$ (Lemma~\ref{largebetalem}). That is, with $\kappa$, $\tau$ and $\mu_0$ as in \eqref{kappataumu}, we have that \eqref{tildeNeq} reduces to
    \begin{equation}\label{tildeNeqbeta}
        \widetilde{N}(\zeta,\beta_{d(k)})=(1+\zeta)(2k+\varepsilon)-1-d(k)-\{2k+\varepsilon-d(k)\}\frac{\zeta}{\beta_{d(k)}}+\left\{d(k)+\zeta\right\}\log\left(\frac{\beta_{d(k)}}{\zeta}\right),
    \end{equation}
    where, as before, we set $\varepsilon=10^{-10}$. 

    For $10<k\leq 100$, we set $\zeta=0.2$ and directly compute \eqref{tildeNeqbeta}. Here, values of $\beta_{d(k)}$ are taken from \cite{bookertable}. For each such $k$, we find that $\widetilde{N}(0.5,\beta_{d(k)})\leq 4k<4k+1$ for even $k$ and $\widetilde{N}(0.5,\beta_{d(k)})\leq 3.1k<4k-1$ for odd $k$ as desired.

    For $k>100$ we set $\zeta=0.1$ and use the bounds in Lemmas \ref{largebetalem} and \ref{divisorlem} to bound $\widetilde{N}(\zeta,\beta_{d(k)})/k$ above by
    \begin{align}\label{Ntildeonk}
        &\frac{1.1}{k}(2k+\varepsilon)-\frac{3}{k}-\{2k+\varepsilon+d^+(k)\}\frac{0.1}{3.75k\:d^+(k)}+\frac{d^+(k)+0.1}{k}\log\left(\frac{3.75\: d^+(k)}{0.1}\right)\notag\\
        &\quad<\frac{1.1}{k}(2k+\varepsilon)+\frac{d^+(k)+0.1}{k}\log\left(37.5\: d^+(k)\right)
    \end{align}
    with $d^+(k)=k^{1.07/\log\log k}$. Now, \eqref{Ntildeonk} is decreasing for $k>100$ and is less than 3.94 in this range, so that we can set $\widetilde{M}(k)=\lceil3.94k\rceil<4k-1$ as desired.\\

    \textbf{Asymptotic value for $\widetilde{M}(k)$}: Finally we prove the asymptotic expressions \eqref{unconasym2} and \eqref{ehasym2} for $\widetilde{M}(k)$. To do so, we again let $\xi=\beta_{\kappa}$ so that $f_{\kappa}(\xi)=0$. Removing all the negative terms in the definition \eqref{tildeNeq} of $\widetilde{N}$, and substituting $\kappa=d(k)$ gives
    \begin{equation}\label{asymNtilde}
        \widetilde{N}(d(k),\mu_0,\tau;\zeta,\beta_{d(k)})<(1+\zeta)\mu_0+\{d(k)+\zeta\}\log\left(\frac{\beta_{d(k)}}{\zeta}\right).
    \end{equation}
    Now, for any $\varepsilon>0$, we can set $\mu_0=2k+\varepsilon$ (Lemma \ref{m0lem}), $d(k)\ll k^{\varepsilon}$ (Lemma \ref{divisorlem}) and $\beta_{d(k)}\ll d(k)$ (Lemma \ref{largebetalem}), so that
    \begin{equation*}
        \widetilde{N}(d(k),\mu_0,\tau;\zeta,\beta_{d(k)})=(1+\zeta)(2k+\varepsilon)+o_{\zeta}(k).
    \end{equation*}
    Since $\zeta>0$ is a free parameter, a suitable rescaling of $\varepsilon$ (in terms of $\zeta$) yields \eqref{unconasym2}. 

    We now assume the Elliott--Halberstam conjecture and prove \eqref{ehasym2}. Under this assumption, we can set $\tau=1-\delta$ for some arbitrary $\delta>0$ by Lemma \ref{bomvinthm2}. This means that $\mu_0$, appearing in the condition \eqref{m0con}, can be set as
    \begin{equation}\label{mu0eh}
        \mu_0=\frac{k+\delta}{1-\delta}
    \end{equation}
    for sufficiently large $N$. Substituting \eqref{mu0eh} into \eqref{asymNtilde} yields \eqref{ehasym2}, with $\varepsilon$ depending on $\zeta$ and $\delta$.
\end{proof}

\section{Proof of Theorem \ref{mainthmsimple}}\label{simplesect}
In this section we prove Theorem \ref{mainthmsimple}. This will be done using the following result, which is an immediate corollary of Theorem \ref{mainthmtech}.
\begin{corollary}\label{tildeMcor}
    Let $k\geq 1$ and $\widetilde{M}(k)$ be as in Theorem \ref{mainthmtech}. Then, every sufficiently large $N$ can be expressed as
    \begin{equation*}
        N=p^k+\eta
    \end{equation*}
    where $p$ is prime and $\eta$ has at most
    \begin{equation}\label{mtildeeq}
        \widetilde{M}(k)+\sum_{\substack{m\in\{1,2\} \\ 2^{m-1}\mid k}}1+\sum_{\substack{m\geq 3\\ 2^{m-2}\mid k}}1+\sum_{\substack{q\in\mathbb{P}\setminus\{2\}\\
        m\geq 1\\ q^{m-1}(q-1)\mid k}}1
    \end{equation}
    prime factors. Here, $\mathbb{P}$ denotes the set of all primes.
\end{corollary}
For fixed values of $k$, one can use Corollary \ref{tildeMcor} to compute valid values of $M(k)$ in Theorem \ref{mainthmsimple}. For $1\leq k\leq 10$ this is done in Table \ref{Mktable}. Notably, for odd $k$, one can simply set $M(k)=\widetilde{M}(k)+1$.
\begin{proof}[Proof of Theorem \ref{mainthmsimple}]
    If $k\geq 3$ is odd, then we can set $M(k)=\widetilde{M}(k)+1=4k$ by Corollary \ref{tildeMcor}. The asymptotic results \eqref{unconasym} and \eqref{ehasym} also follow similarly (for odd $k$). So, we assume that $k$ is even and bound each sum in \eqref{mtildeeq}.

    For even $2\leq k\leq 10$, we manually compute \eqref{mtildeeq}, giving the values $M(k)\leq 6k$ that appear in Table \ref{Mktable}. For $k>10$, we first let
    \begin{equation*}
        S(k):=\sum_{\substack{m\in\{1,2\} \\ 2^{m-1}\mid k}}1+\sum_{\substack{m\geq 3\\ 2^{m-2}\mid k}}1+\sum_{\substack{q\in\mathbb{P}\setminus\{2\}\\
        m\geq 1\\ q^{m-1}(q-1)\mid k}}1
    \end{equation*}
    be the three sums in \eqref{mtildeeq}. If $10<k\leq 100$, a short computation gives
    \begin{equation*}
        S(k)\leq 2k-1,
    \end{equation*}
    so that we can set $M(k)=6k$ for each even $k\leq 100$. Now suppose that $k>100$. To begin with,
    \begin{equation}\label{sumbound1}
        \sum_{\substack{m\in\{1,2\} \\ 2^{m-1}\mid k}}1\leq 2.
    \end{equation}
    Then,
    \begin{equation}\label{sumbound2}
        \sum_{\substack{m\geq 3\\ 2^{m-2}\mid k}}1\leq\frac{\log k}{\log 2}.
    \end{equation}
    For the third sum in $S(k)$, note that we are summing over $q$ such that $q-1\mid k$. The number of such $q$ is less than $d(k)$, the number of divisors of $k$. Hence,
    \begin{equation}\label{sumbound3}
        \sum_{\substack{q\in\mathbb{P}\setminus\{2\}\\ m\geq 1\\ q^{m-1}(q-1)\mid k}}1\leq d(k)\sum_{1\leq m\leq (\log k/\log 3)+1}1\leq d(k)\left(\frac{\log k}{\log 3}+1\right).
    \end{equation}
    Combining \eqref{sumbound1}, \eqref{sumbound2}, \eqref{sumbound3} and Lemma \ref{divisorlem}, we have
    \begin{equation}\label{sbound}
        S(k)\leq 2+\frac{\log k}{\log 2}+k^{1.07/\log\log k}\left(\frac{\log k}{\log 3}+1\right).
    \end{equation}
    From \eqref{sbound} we see that $S(k)=o(k)$ so that the asymptotic expressions \eqref{unconasym} and \eqref{ehasym} for $M(k)$ follow from the analogous expressions \eqref{unconasym2} and \eqref{ehasym2} for $\widetilde{M}(k)$. Moreover, we see that for $k>100$, $S(k)<1.4k<2k$ so that
    \begin{equation*}
        M(k)=6k
    \end{equation*}
    is valid for \emph{all} even $k\geq 2$ as required.
\end{proof}

\section{Further discussion and conjectures}\label{furthersect}

\subsection{A conjectural value for $\widetilde{M}(k)$}
Here we attempt to conjecture (within reason) a better value for $\widetilde{M}(k)$ in Theorem \ref{mainthmtech}, thereby giving an indication of how far our results may be improved. When $k=1$, extensive computational evidence (e.g.\ \cite{oliveira2014empirical}) suggests that $\widetilde{M}(1)=1$ is possible. For $k\geq 2$, an optimal value of $\widetilde{M}(k)$ is less clear. One natural extension of Goldbach's conjecture would be to guess that $\widetilde{M}(k)=k$ is valid for large enough $N$, which is ``just" better than our conditional expression \eqref{ehasym2} for large $k$. In fact we do conjecture this to be the case.

\begin{conjecture}\label{kconj}
    One can take $\widetilde{M}(k)=k$ for all $k\geq 1$ in Theorem \ref{mainthmtech}. 
\end{conjecture}

Corresponding conjectural values of $M(k)$ in Theorem \ref{mainthmsimple} can then be obtained using Corollary \ref{tildeMcor}. A particularly neat case of Conjecture \ref{kconj} is as follows.

\begin{conjecture}\label{kconj2}
    Let $k\geq 1$ be odd and $N$ be a sufficiently large even number. Then, $N$ can be expressed as the sum of a $k$th power of a prime and a number with at most $k$ prime factors. 
\end{conjecture}

Conjecture \ref{kconj} agrees well with computation. In particular, we found via a simple computation that for $k\in\{2,3,4,5\}$, every $N$ with $N_k\leq N\leq 10^8$ can be written in the form \eqref{secondNeq} with $\widetilde{M}(k)=k$, where $N_2=250$, $N_3=28$, $N_4=125522$ and $N_5=137876$. Although, if we lowered the value of $\widetilde{M}(k)$ any further, far more counterexamples for $N$ were found, which do not appear to taper off as $N$ grows. For example, the five largest odd $N\leq 10^8$ that cannot be expressed as $N=p^3+2\eta$ with $\Omega(\eta)\leq 2$ are
\begin{equation*}
    N=91\: 159\: 823,\ 97\: 544\: 663,\ 97\: 761\: 383,\ 98\: 674\: 883,\ 98\: 924\: 363.
\end{equation*}

However, besides computational evidence, there does not appear to be any obvious reason why $\widetilde{M}(k)<k$ is not possible. One cannot take $\widetilde{M}(k)$ too small though, as we demonstrate via the following proposition. 

\begin{proposition}\label{cycloprop}
    Fix $k\geq 2$ and let $S_k$ be the set of positive integers $N$ such that $N\not\equiv 1\pmod{q}$ for all primes $q$ with $q-1\mid k$. Then, there exists infinitely many $N\in S_k$ such that for all $p<N^{1/k}$
    \begin{equation*}
        \Omega(N-p^k)\geq d(k),
    \end{equation*}
    where $d(k)$ is the number of divisors of $k$.
\end{proposition}
\begin{proof}
    Let
    \begin{equation*}
        K=\prod_{\substack{q\ \text{prime}\\q-1\mid k}}q
    \end{equation*}
    and 
    \begin{equation*}
        \Phi_d(x)=\prod_{\substack{1\leq a<d\\(a,d)=1}}(x-\exp(2\pi ia/d))
    \end{equation*}
    denote the $d$th cyclotomic polynomial. Consider any $M\equiv 0\pmod{K}$ with $M-1$ not prime. Infinitely many $M$ exist of this form, for otherwise it would contradict the prime number theorem. For a choice of such $M$, let $N=M^k\in S_k$. Then, for any $p<M=N^{1/k}$,
    \begin{equation}
        N-p^k=p^k\left(\left(\frac{M}{p}\right)^k-1\right)=p^k\prod_{d\mid k}\Phi_d\left(\frac{M}{p}\right)=\prod_{d\mid k}p^{\varphi(d)}\Phi_d\left(\frac{M}{p}\right),\label{mkpkeq}
    \end{equation}
    where we have used the cyclotomic decomposition of $x^k-1$ \cite[Proposition 13.2.2]{ireland1990classical} and the elementary identity
    \begin{equation*}
        \sum_{d\mid k}\varphi(d)=k.
    \end{equation*}
    Since each $\Phi_d(x)$ has degree $\varphi(d)$ and $\Phi_d(x)\in\mathbb{Z}[x]$ \cite[p. 194]{ireland1990classical}, we then have
    \begin{equation*}
        p^{\varphi(d)}\Phi_d\left(\frac{M}{p}\right)\in\mathbb{Z}.
    \end{equation*}
    To conclude, it suffices to show that each factor in \eqref{mkpkeq} satisfies 
    \begin{equation*}
        \left| p^{\varphi(d)}\Phi_d\left(\frac{M}{p}\right)\right|>1.
    \end{equation*}
    This follows from the reverse triangle inequality:
    \begin{equation*}
        \left|p^{\varphi(d)}\Phi_d\left(\frac{M}{p}\right)\right|=\left|\prod_{\substack{1\leq a\leq d\\(a,d)=1}}\left(M-p\cdot \exp\left(\frac{2\pi i a}{d}\right)\right)\right|\geq\prod_{\substack{1\leq a\leq d\\(a,d)=1}}\left|M-p\right|>1
    \end{equation*}
    since $p<M$ and $M-1$ is not prime.
\end{proof}
\begin{remark}
    A similar argument could be used to prove an analogous statement for other residue classes of $N$. However, we only considered $N\in S_k$ here for simplicity as it corresponds to the case where $m_1=m_2=\cdots=m_{\ell_k}=0$ in Theorem \ref{mainthmtech}.
\end{remark}
Proposition \ref{cycloprop} shows that one always requires $\widetilde{M}(k)\geq d(k)$ in Theorem \ref{mainthmtech}. Note that for large $k$ this lower bound is far from our conjectured value $\widetilde{M}(k)=k$ since $d(k)=O(k^{\varepsilon})$ for any $\varepsilon>0$ (see Lemma \ref{divisorlem}).

\subsection{The case where $N$ is not a $k$th power}
As in the proof of Proposition \ref{cycloprop}, setting $N$ to be a $k$th power is a natural choice to increase the number of prime factors of $N^k-p^k$. Thus, if we restrict to $N$ not being a $k$th power, we expect that $\widetilde{M}(k)$ could be lowered. Certainly, one could attempt to rework our proof of Theorem \ref{mainthmtech} in this case. A key observation is that Proposition \ref{omegaprop} may be sharpened since if $N$ is not a $k$th power, then there is a non-zero proportion of primes $p$ for which $N$ is not a $k$th power residue mod $p$. This follows from the Chebotarev density theorem (or the weaker Frobenius density theorem). For example, if $k=2$ and $N$ not a square, then $N$ is only a quadratic residue mod $p$ for (asymptotically) $50\%$ of primes. Arguing as in the proof of Proposition \ref{omegaprop} then gives  
\begin{equation*}      
    \prod_{\substack{z_1\leq p<z_2\\ p\in\mathcal{P}}}\left(1-\frac{g(p)}{p}\right)^{-1}<\frac{\log z_2}{\log z_1}\left\{1+O_N\left(\frac{1}{\log z_1}\right)\right\},\qquad z_2>z_1\geq 2,
\end{equation*}
which suggests using a sieve of dimension $\kappa=1$. Problematically though, the error term arising from the Chebotarev density theorem depends on $N$ (see e.g.\ \cite{thorner2019unified}) so the DHR sieve can not be applied as in the proof of Theorem \ref{mainthmtech}.

Inspired by this observation, we make a further conjecture in the case when $k=2$ and $N$ is not a square. For simplicity we only consider $N\equiv0,2\pmod{6}$. That is, $N\not\equiv 1\pmod{2}$ and $N\not\equiv 1\pmod{3}$ so the coefficient of $\eta$ in \eqref{secondNeq} is 1.
\begin{conjecture}\label{nonsquarecon}
    Let $N>24476$ be a non-square such that $N\equiv 0,2\pmod{6}$. Then, $N$ can be written as
    \begin{equation}\label{Np2qeq}
        N=p^2+q,
    \end{equation}
    where $p$ and $q$ are primes.
\end{conjecture}

By Hua's theorem (Theorem \ref{huathm}) almost all $N\equiv 0,2\pmod{6}$ can be represented as \eqref{Np2qeq}. A non-explicit version of Conjecture \ref{nonsquarecon} has also previously been conjectured in the literature \cite{kumchev2009sums,penevarecent}. However, the crucial fact that $N$ must not be a square is often omitted.

Note that Conjecture \ref{nonsquarecon} is a strengthening of Hardy and Littlewood's ``Conjecture H" for $N\equiv 0,2\pmod{6}$. Namely, in \cite{hardy1923some}, Hardy and Littlewood conjectured that every large non-square integer $N$ can be written as 
\begin{equation}\label{hardylittleeq}
    N=n^2+q
\end{equation}
for some \emph{integer} $n$ and prime $q$. 

Using a simple computation, we verified Conjecture \ref{nonsquarecon} for $N\leq 10^8$. This computation is also how we arrived at the lower bound $N>24476$. We remark that an analogous computation suggests an optimal lower bound of $N>21679$ for Hardy and Littlewood's Conjecture H.

\section*{Acknowledgements}
We appreciate the fruitful discussions with Adrian Dudek, Michael Harm, Bryce Kerr, Igor Shparlinski, Valeriia Starichkova, Timothy Trudgian and John Voight during the composition of this work. We also thank the anonymous referees for their valuable feedback and suggestions. Both authors' research was supported by an Australian Government Research Training Program (RTP) Scholarship.

\printbibliography

\end{document}